\documentclass[a4paper,12pt]{article}
\usepackage{a4wide}
\usepackage{amsmath}
\usepackage{amssymb}
\usepackage{amsthm}
\usepackage{latexsym}
\usepackage{graphicx}
\usepackage[english]{babel}
\usepackage{makeidx}
\usepackage{mathtools}

\newtheorem{obs} [subsection]{Remark}

\newtheorem{prop}[subsection]{Proposition}

\newtheorem{teor}[subsection]{Theorem}

\newtheorem{cor} [subsection]{Corollary}
\DeclarePairedDelimiter\ceil{\lceil}{\rceil}

\newcommand{\pa}{p_{\mathbf a}}

\def\supp{\operatorname{supp}}

\def\Ker{\operatorname{Ker}}
\def\Gal{\operatorname{Gal}}

\def\ord{\operatorname{ord}}
\numberwithin{equation}{section}

\begin{document}
\selectlanguage{english}
\frenchspacing

\large
\begin{center}
\textbf{On the semigroup ring of holomorphic Artin L-functions}

Mircea Cimpoea\c s 
\end{center}
\normalsize

\begin{abstract}
Let $K/\mathbb Q$ be a finite Galois extension and let $\chi_1,\ldots,\chi_r$ be the irreducible characters of the Galois group $G:=Gal(K/\mathbb Q)$.
Let $f_1:=L(s,\chi_1),\ldots,f_r:=L(s,\chi_r)$ be their associated Artin L-functions. For $s_0\in \mathbb C\setminus\{1\}$, we denote
$Hol(s_0)$ the semigroup of Artin $L$-functions, holomorphic at $s_0$.
Let $\mathbb F$ be a field with $\mathbb C \subseteq \mathbb F \subseteq \mathcal M_{<1}:=$ the field of meromorphic functions of order $<1$.
We note that the semigroup ring $\mathbb F[Hol(s_0)]$ is isomorphic to a toric ring 
$\mathbb F[H(s_0)]\subseteq \mathbb F[x_1,\ldots,x_r]$, where $H(s_0)$ is an affine subsemigroup of $\mathbb N^r$ 
minimally generated by at least $r$ elements, and we describe $\mathbb F[H(s_0)]$ when the toric ideal $I_{H(s_0)}=(0)$. 
Also, we describe $\mathbb F[H(s_0)]$ and $I_{H(s_0)}$ when $f_1,\ldots,f_r$ have only simple zeros and simple poles at $s_0$.


\end{abstract}

\section*{Introduction}

Let $K/\mathbb Q$ be a finite Galois extension. 
For the character $\chi$ of a representation of the Galois group $G:=Gal(K/\mathbb Q)$
on a finite dimensional complex vector space, let $L(s,\chi):=L(s,\chi,K/\mathbb Q)$ be the corresponding Artin L-function (\cite[P.296]{artin2}). 
Artin conjectured that $L(s,\chi)$ is holomorphic in $\mathbb C\setminus \{1\}$ and $s=1$ is a simple pole. Brauer \cite{brauer}
proved that $L(s,\chi)$ is meromorphic in $\mathbb C$, of order $1$.
Let $\chi_1,\chi_2,\ldots,\chi_r$ be the irreducible characters of $G$ with the dimensions $d_1:=\chi_1(1)$ ,$\ldots, d_r=\chi_r(1)$.
Let $f_1:=L(s,\chi_1),\ldots,f_r:=L(s,\chi_r)$. Artin \cite[Satz 5, P. 106]{artin1} proved that $f_1,\ldots,f_r$ are multiplicatively independent.
It was proved in \cite{lucrare} that $f_1,\ldots,f_r$ are algebraically independent over $\mathbb C$. This result was extended in \cite{forum}
to the field $\mathcal M_{<1}$ of meromorphic functions of order $<1$. 

For two characters $\varphi$ and $\psi$ of $G$ it holds that $L(s,\varphi+\psi)=L(s,\varphi)+L(s,\psi)$, so the set of $L$-functions
corresponding to all characters of $G$ is a multiplicative semigroup $Ar$. Let $s_0\in \mathbb C\setminus \{1\}$ and let $Hol(s_0)$ be
the subsemigroup of $Ar$ consisting of the $L$-functions which are holomorphic at $s_0$. Artin's conjecture at $s_0$ is 
$$Hol(s_0)=Ar.$$
It was proved in \cite{numb} that $Hol(s_0)$ is an affine semigroup isomorphic to a subsemigroup $H(s_0)$ of $\mathbb N^r$. 
Let $\mathbb F$ be a field such that $\mathbb C \subseteq \mathbb F \subseteq \mathcal M_{<1}$.
From  \cite[Theorem 1]{numb} and \cite[Corollary 9]{forum}, we have the $\mathbb F$-algebra isomorphism $\mathbb F[Hol(s_0)] \cong \mathbb F[H(s_0)]$,
therefore, the Artin's conjecture at $s_0$ can be stated as
$$\mathbb F[H(s_0)] = \mathbb F[x_1,\ldots,x_r].$$
Assuming that $\mathbb F[H(s_0)]$ is minimally generated, as a $\mathbb F$-algebra, 
by a set of monomials $\{u_1,\ldots,u_m\}$, we consider the canonical epimorphism
$$\Phi:\mathbb F[t_1,\ldots,t_m] \longrightarrow \mathbb F[H(s_0)],\;\Phi(t_i)=u_i,\;1\leq i\leq m,$$
The ideal $I_{H(s_0)}:=\Ker(\Phi)$ is called the \emph{toric ideal} of $\mathbb F[H(s_0)]$, see \cite{villa} for further details.

In Proposition $1.1$, we prove that if $s_0$ is a zero of $f_1$ and is not a zero of $f_2,\ldots,f_r$, then Artin conjecture holds at $s_0$
if i) $d_1=1$ or ii) $d_1=2$ and $s_0$ is simple, i.e. $f_1'(s_0)\neq 0$. In Proposition $1.2$, we prove that $\mathbb F[H(s_0)]$ is minimally 
generated by at least $r$ monomials and, therefore, $I_{H(s_0)}=(0)$ if and only if $\mathbb F[H(s_0)]$ 
is minimally generated by $r$ monomials. 

For $1\leq j\leq r$, let $\ell_j$ be the order of $s_0$ as a zero of $f_j$.
In Theorem $1.4$, we prove that $I_{H(s_0)}=(0)$ if and only if $s_0$ is a zero of $f_1$ and is not a zero of $f_2,\ldots,f_{r-1}$
and $\ell_1|\ell_j$, $(\forall)2\leq j\leq r$. In Remark $1.5$, we describe the Hilbert series of $\mathbb F[H(s_0)]$ when $I_{H(s_0)}=(0)$.

For any $1\leq t\leq r-1$ we let
$$N_t:=\min\{\binom{r-1}{t-1}+\binom{r-2}{t}+1,\binom{r}{t}\} \text{ and }$$
$$L_t:=|\{\supp(v)\;:\;v\in \mathbb F[H(s_0)] \text{ monomial },\;|\supp(v)|=t\}|.$$
where $\supp(v):=\{x_j\;:\;x_j|v\}$ is the \emph{support} of the monomial $v$.
In Theorem $1.6$ we prove that Artin Conjecture is true at $s_0$ if and only if $I_{H(s_0)}=(0)$ and
there exists $1\leq t\leq r-1$ such that $L_t\geq N_t$.

In Corollary $1.7$, we reprove the main results of \cite{numb} and \cite{numb2}. In Corollary $1.8$
 we prove that if the zeros of $f_1,\ldots,f_r$ are simple
and distinct, Artin's conjecture holds if and only if $\mathbb F[H(s_0)]$ satisfies several equivalent conditions.
Under the same hypothesis, 
using a result from \cite{monat}, in Corollary $1.10$ we prove that Artin's conjecture holds if 
the group $G$ is \emph{almost monomial} (as it was defined in \cite{monat}). 

In Proposition $1.11$, we describe $\mathbb F[H(s_0)]$ when the zeros at $s_0$ of $f_1,\ldots,f_r$, if any, are simple.
In Proposition $1.12$, we describe $\mathbb F[H(s_0)]$ when the poles at $s_0$ of $f_1,\ldots,f_r$, if any, are simple.
In Theorem $1.13$ we describe $\mathbb F[H(s_0)]$ and $I_{H(s_0)}$ if $f_1,\ldots,f_r$ have at most simple zeros and simple poles.

A \emph{virtual character} $\chi$ of the group $G=\Gal(K/\mathbb Q)$ is a formal sum
$$\chi := a_1\chi_1 + a_2\chi_2 + \cdots + a_r\chi_r,\;a_j\in\mathbb Z,\;1\leq j\leq r.$$
We associate to $\chi$ the meromorphic function
$L(s,\chi):=f_1^{a_1}f_2^{a_2}\cdots f_r^{a_r}$. The theory of virtual characters is an useful tool in studying L-Artin functions,
see for instance \cite{amit}, \cite{rhoades} and \cite{wong}.

We note that $\overline{Ar}:=\{L(s,\chi)\;:\;\chi \text{ a virtual character of }G \}  $
is a multiplicative group, isomorphic to $(\mathbb Z^r,+)$.
We let $\overline{Hol(s_0)}:=\{L(s,\chi)\in\overline{Ar}\;:\; L(s,\chi)$ holomorphic in $s_0\}$ and we prove that 
$\overline{Hol(s_0)}$ is an affine monoid isomorphic to a submonoid $\overline{H(s_0)}\subset \mathbb Z^r$. In Proposition $2.2$
and Proposition $2.3$ we describe the toric ring $\mathbb F[\overline{H(s_0)}] \subset \mathbb F[x_1^{\pm 1},\ldots,x_r^{\pm 1}]$ when $I_{H(s_0)}=(0)$, respectively when $f_1,\ldots,f_r$
have simple zeros and poles.

\section{Main results}

Using the notations from the Introduction, from \cite[Theorem 1]{numb} and \cite[Corollary 9]{forum}, it follows that the map
\begin{equation}\label{lab1}
 \mathbb F[x_1,\ldots,x_r] \longrightarrow \mathbb F[f_1,\ldots,f_r],\;x_i\mapsto f_i,\;(\forall)1\leq i\leq r.
\end{equation}
is an isomorphism of $\mathbb F$-algebras.
If $\chi$ is a character of $G$, then there exists some nonnegative integers $a_1,\ldots,a_r$ such that $\chi=a_1\chi_1+\cdots+a_r\chi_r$.
It hold that \cite[P. 297, (10)]{artin2}
$$L(s,\chi) = f_1^{a_1}\cdots f_r^{a_r}.$$
Let $s_0\in\mathbb C\setminus\{1\}$. We let $Hol(s_0)$ be the semigroup consisting of the $L$-functions which are holomorphic at $s_0$.
It was proved in \cite[Theorem 1]{numb} that $Hol(s_0)$ is an affine semigroup isomorphic to a subsemigroup $H(s_0)$ of $\mathbb N^r$.
More precisely, the isomorphism is given by
\begin{equation}\label{lab2}
Hol(s_0) \longrightarrow H(s_0),\; f=f_1^{a_1}\cdots f_r^{a_r}\mapsto (a_1,\ldots,a_r). 
\end{equation}
According to (\ref{lab1}) and (\ref{lab2}), it follows that
$\mathbb F[H(s_0)] \subseteq \mathbb F[x_1,\ldots,x_r]$ is a toric ring. Moreover, the Artin Conjecture is equivalent to 
\begin{equation}
\mathbb F[H(s_0)] = \mathbb F[x_1,\ldots,x_r],\;(\forall)s_0\in\mathbb C\setminus\{1\}. 
\end{equation}
Note that the (minimal) generators of the semigroup $H(s_0)$ corespond to the (minimal) monomial generators of the $\mathbb F$-algebra, $\mathbb F[H(s_0)]$.

For a meromorphic function $f$, we denote $\ord_{s=s_0}f$ the order of $s_0$ as a zero of $f$. By abuse of notation, we denote 
$$\ell_j:=\ord(x_j):=\ord_{s=s_0}f_j, \; 1\leq j\leq r, \text{ and }$$
$$\ord(x_1^{a_1}\cdots x_r^{a_r}):=a_1\ord(x_1)+\cdots +a_r\ord(x_r).$$
Since a $L$-function $f=L(s,\chi)$ is holomorphic at $s_0$, if and only if $\ord_{s=s_0}f\geq 0$, \
it follows that a monomial $u$ is in $\mathbb F[H(s_0)]$ if and only if $\ord(u)\geq 0$. Hence
\begin{equation}\label{lab4}
\mathbb F[H(s_0)] = \mathbb F[x_1^{a_1}\cdots x_r^{a_r}\;:\;a_1\ell_1+\cdots+a_r\ell_r\geq 0].
\end{equation}
The character $\chi:=d_1\chi_1+\cdots+d_r\chi_r$ is called the \emph{regular character} of $G$.
It holds that $$\zeta_K(s):=L(s,\chi)=f_1^{d_1}\cdots f_r^{d_r},$$ the Dedekind zeta function of $K$, is holomorphic at $s_0$, hence 
\begin{equation}\label{cur1}
u:=x_1^{d_1}\cdots x_r^{d_r}\in \mathbb F[H(s_0)].
\end{equation}
Now, assume that the Artin conjecture fails at $s_0$, i.e. at least one of the $\ell_j$'s is negative.
Since $x_1^{d_1}\cdots x_r^{d_r}\in \mathbb F[H(s_0)]$, it follows that at least one of the $\ell_j$'s is positive.
By reordering the characters $\chi_1,\ldots,\chi_r$, we can assume that there exists $1\leq p < q\leq r$ such that 
\begin{equation}\label{lab5}
\ell_1,\ldots,\ell_p>0,\;\ell_{p+1},\ldots,\ell_{q}<0,\;\ell_{q+1}=\cdots=\ell_r=0.
\end{equation}
From (\ref{lab4}) and (\ref{lab5}) it follows that
\begin{equation}\label{lab6}
 \mathbb F[H(s_0)] = \mathbb F[x_1^{a_1}\cdots x_r^{a_r}\;:\;a_1\ell_1+\cdots+a_p\ell_p \geq  a_{p+1}|\ell_{p+1}|+\cdots+a_q|\ell_q| ].
\end{equation}
By a theorem of Stark (\cite[Theorem 3]{stark}), if $\ord(u)=\ord_{s=s_0}(\zeta_K)\leq 1$, then $f_1,\ldots,f_r$ are holomorphic at $s_0$,
a contradiction. Therefore, from  (\ref{cur1}) and (\ref{lab5}) it follows that
\begin{equation}\label{cur3}
d_1\ell_1+\cdots+d_p\ell_p \geq d_{p+1}|\ell_{p+1}| + \cdots + d_{q}|\ell_q| + 2.
\end{equation}
Moreover, according to a theorem of Rhoades (\cite[Theorem 2]{rhoades}), it follows that 
\begin{equation}\label{cur2}
u_j^{\pm}:=x_j^{\pm 1}u \in \mathbb F[H(s_0)],\;(\forall)1\leq j\leq r. 
\end{equation}
From (\ref{cur1}), (\ref{lab5}) and (\ref{cur2}) it follows that
\begin{equation}\label{labX}
d_1\ell_1 + \cdots + d_p \ell_p \geq d_{p+1}|\ell_{p+1}| + \cdots + d_q|\ell_q| + \max\{|\ell_j|\;: 1\leq j\leq r\}.
\end{equation}

\begin{prop}
Artin conjecture holds at $s_0$ in the following cases:
\begin{enumerate}
\item[(1)] $s_0$ is a zero of $f_1$ and is not a zero of $f_2,\ldots,f_r$ and $d_1=1$.
\item[(2)] $s_0$ is a simple zero of $f_1$ and is not a zero of $f_2,\ldots,f_r$ and $d_1=2$. 
\end{enumerate}
\end{prop}

\begin{proof}
(1) Assume that Artin conjecture fails at $s_0$.
According to (\ref{labX}), it follows that
$\ell_1 \geq d_{p+1}|\ell_{p+1}| + \cdots + d_q|\ell_q| + \ell_1 > \ell_1$,
a contradiction.

(2) Assume that Artin conjecture fails at $s_0$.
    According to (\ref{cur3}), it follows that
$$2 = d_1 \geq d_2|\ell_2| + \cdots + d_r|\ell_r| + 2 > 2,$$
a contradiction. The conclusion follows also from \cite[Theorem 1]{monat}.
\end{proof}

\begin{obs}
\emph{According to a a result of Amitsur (\cite[Theorem 1]{amit}) $d_1,\ldots,d_r\leq 2$ if and only if either: 
i) $G$ is abelian, ii) $G$ has an abelian subgroup of order $2$ or ii) $G/Z(G)$ is an abelian $2$-group of order $8$, where $Z(G)$ is the center of $G$.
In this context, if $f_1,\ldots,f_r$ have simple zeros distinct, then, according to Proposition $1.1$, Artin conjecture holds.}
\end{obs}

\begin{prop} We have that:
\begin{enumerate}
\item[(1)] The minimal number of generators of $\mathbb F[H(s_0)]$ is $r$. 
\item[(2)] $I_{H(s_0)}=(0)$ if and only if $\mathbb F[H(s_0)]$ is minimally generated by $r$ monomials.
\end{enumerate}
\end{prop}

\begin{proof}
(1) We use a similar argument as in the first part of the proof of \cite[Theorem 1]{numb}. If Artin conjecture holds, i.e. 
$\mathbb F[H(s_0)] = \mathbb F[x_1,\ldots,x_r]$, then the assertion is obvious.
Assume this is not the case. Let $1\leq p< q\leq r$ as in (\ref{lab5}). For any $p+1\leq j\leq q$, we let  
$m_j:=\ceil{\frac{-\ell_j}{\ell_1}}$. It follows that
$x_1^{m_j}x_j \in \mathbb F[H(s_0)],\; x_1^{m_j-1}x_j \notin \mathbb F[H(s_0)],\;(\forall) p+1\leq j\leq q$.
Therefore,  $x_1,\ldots,x_p,x_1^{m_{p+1}}x_{p+1},\ldots,x_1^{m_{q}}x_{q}, x_{q+1},\ldots,x_r$ are minimal monomial 
generators of $\mathbb F[H(s_0)]$. Hence $\mathbb F[H(s_0)]$ is minimally generated by at least $r$ monomials.

(2) If $\mathbb F[H(s_0)] = \mathbb F[x_1,\ldots,x_r]$, then there is nothing to prove. If this is not the case, as above,
$x_1,\ldots,x_p,x_1^{m_{p+1}}x_{p+1},\ldots,x_1^{m_{q}}x_{q}, x_{q+1},\ldots,x_r$ is part of the set of minimal generators of 
$\mathbb F[H(s_0)]$. If $\mathbb F[H(s_0)]$ is minimally generated by $r$ monomials, it follows that
$$\mathbb F[H(s_0)] = \mathbb F[x_1,\ldots,x_p,x_1^{m_{p+1}}x_{p+1},\ldots,x_1^{m_{q}}x_{q}, x_{q+1},\ldots,x_r] \cong \mathbb F[t_1,\ldots,t_r],$$
thus $I_{H(s_0)}=(0)$. Conversely, if the number of minimal monomial generators of $\mathbb F[H(s_0)]$ is $>r$, then $I_{H(s_0)}\neq (0)$, otherwise
$\mathbb F[H(s_0)]$ which is a $\mathbb F$-subalgebra of $\mathbb F[x_1,\ldots,x_r]$ would be isomorphic with a polynomial algebra in $>r$ indeterminates, a contradiction.
\end{proof}

\begin{teor}
If Artin conjecture does not hold at $s_0$ then the following are equivalent:
\begin{enumerate}
\item[(1)] $s_0$ is a simple zero of $f_1$ and is not a zero of $f_2,\ldots,f_r$ and $\ell_1 | \ell_j$, $(\forall)2\leq j\leq r$.
\item[(2)] $I_{H(s_0)}=(0)$.
\end{enumerate}
Moreover, if one the above conditions holds, there exists $2\leq q\leq r$ such that
$$\mathbb F[H(s_0)] = \mathbb F[x_1,x_1^{m_{2}}x_{2},\ldots, x_1^{m_q}x_q,x_{q+1},\ldots,x_r],$$
where $m_j:=-\frac{\ell_j}{\ell_1}$, $2\leq j\leq q$.
\end{teor}

\begin{proof}
$(1)\Rightarrow (2)$ According to (\ref{lab5}), there exists
$1 < q \leq r$ such that $f_2,\ldots,f_q$ have poles at $s_0$. Let $m_j:=-\frac{\ell_j}{\ell_1}$, $2\leq j\leq q$.
It follows that 
$$\ord(x_1^{m_j}x_j)=0,\; (\forall)2\leq j\leq q,$$
hence $x_1^{m_j}x_j\in \mathbb F[H(s_0)],  (\forall)2\leq j\leq q$. We prove that 
\begin{equation}\label{kaka}
\mathbb F[H(s_0)] = \mathbb F[x_1,x_1^{m_{2}}x_{2},\ldots, x_1^{m_q}x_q,x_{q+1},\ldots,x_r]=:R.
\end{equation}
If $u=x_1^{a_1}x_2^{a_2}\cdots x_r^{a_r}\in \mathbb F[H(s_0)]$, then 
$$\ord(u) = \ell_1 a_1 - \ell_{2}a_{2} - \cdots  - \ell_q a_q \geq 0,$$
therefore $$u = x_1^{\ord(u)}(x_{2}x_1^{m_{2}})^{a_{2}} \cdots (x_{q}x_1^{m_{q}})^{a_{q}}x_{q+1}^{a_{q+1}}\cdots x_r^{a_r}
\in R,$$ as required. From \eqref{kaka} and Proposition $1.3(2)$ it follows that $I_{H(s_0)}=(0)$.

$(2)\Rightarrow (1)$ We use a similar argument as in the proof of \cite[Theorem 2]{numb}. 
If $I_{H(s_0)}=(0)$ then, as in the proof of Proposition $1.3(2)$, it follows that
$$\mathbb F[H(s_0)] = \mathbb F[x_1,\ldots,x_p,x_1^{m_{p+1}}x_{p+1},\ldots,x_1^{m_{q}}x_{q}, x_{q+1},\ldots,x_r]=:R,$$
where $m_j:=\ceil{\frac{-\ell_j}{\ell_1}}$, $p+1\leq j\leq q$. 
If $p\geq 2$, i.e. $f_2(s_0)=0$, then by letting $n_j:=\ceil{\frac{-\ell_j}{\ell_2}}$, $p+1\leq j\leq q$, 
it follows that $x_2^{n_{p+1}}x_{p+1},\ldots,x_2^{n_{q}}x_q \in \mathbb F[H(s_0)]$, a contradiction.
If $p=1$ and $\ell_1\nmid \ell_2$, then $\ord(x_1^{m_1}x_2)>0$. On the other hand $\ord(x_1^{-\ell_2}x_2^{\ell_1}) = 0$, hence
$x_1^{-\ell_2}x_2^{\ell_1}\in \mathbb F[H(s_0)]$ and $x_1^{-\ell_2}x_2^{\ell_1}\notin R$, a contradiction.
\end{proof}

\begin{obs}
\emph{
Assume that $I_{H(s_0)}=(0)$ and the Artin conjecture does not hold at $s_0$. 
According to Theorem $1.4$, there exists $2\leq q\leq r$ such that
$$\mathbb F[H(s_0)] = \mathbb F[x_1,x_1^{m_{2}}x_{2},\ldots, x_1^{m_q}x_q,x_{q+1},\ldots,x_r],$$
where $m_j = -\frac{\ell_j}{\ell_1}$, $2\leq j\leq q$. It follows that the map \small{
\begin{equation}\label{kaaka}
\Phi: \mathbb F[t_1,\ldots,t_r] \longrightarrow \mathbb F[H(s_0)],\; \Phi(t_j) := x_1^{m_j}x_j,\;2\leq j\leq q,\;
 \Phi(t_j):=x_j,\; j=1 \text{ and } q+1\leq j\leq r,
\end{equation}}
is a graded isomorphism of $\mathbb F$-algebras, where $\mathbb F[t_1,\ldots,t_r]$ has the (non-standard) grading 
given by $\deg(t_j)=m_j+1$, $2\leq j\leq q$, $\deg(t_j)=1$, $j=1$ and $q+1\leq j\leq r$.
Moreover, from (\ref{cur3}) and (\ref{labX}) it follows that
$d_1 \geq d_2m_2 + \cdots +d_qm_q+\max\{1,\lceil \frac{2}{\ell_1} \rceil\}$.}

\emph{
Let $\mathbf a = (a_1,\ldots,a_r)$ be a vector of positive integers and let $n\geq 0$ be an integer.
$$p_{\mathbf a}(n):=\text{ the number of integer solutions of } a_1x_1+\cdots a_rx_r =n \text{ with } x_i\geq 0, 
1\leq i\leq r, $$
is called the \emph{restricted partition function} associated to $\mathbf a$.}

\emph{
It is well known that if $S=K[t_1,\ldots,t_r]$ is the ring of polynomials with the grading given by $\deg(t_j)=a_j$, $1\leq j\leq r$,
then the \emph{Hilbert function} of $S$ is 
\begin{equation}\label{kaaaa}
 H(S,n)=\pa(n),\;n\geq 0,
\end{equation}
and the \emph{Hilbert series} of $S$ is
\begin{equation}\label{kakaa}
 H_S(t)=\prod_{j=1}^r \frac{1}{1-t^{a_j}}.
\end{equation}
From (\ref{kaaka}), (\ref{kaaaa}) and (\ref{kakaa}), it follows that the Hilbert function and the Hilbert series of $\mathbb F[H(s_0)]$ are:
$$ H(\mathbb F[H(s_0)],n)= p_{(1,m_2+1,\ldots,m_q+1,1,\ldots,1)}(n),\; n\geq 1,$$
$$ H_{\mathbb F[H(s_0)]}(t)=\frac{1}{(1-t^{m_2+1})\cdots (1-t^{m_q+1})(1-t)^{r-q+1}}.$$}
\end{obs}

For any $1\leq t\leq r-1$ we let
$$N_t:=\min\{\binom{r-1}{t-1}+\binom{r-2}{t}+1,\binom{r}{t}\} \text{ and }$$
$$L_t:=|\{\supp(v)\;:\;v\in \mathbb F[H(s_0)] \text{ monomial },\;|\supp(v)|=t\}|.$$
where $\supp(v):=\{x_j\;:\;x_j|v\}$ is the \emph{support} of the monomial $v$.

\begin{teor}
The following are equivalent:
\begin{enumerate}
 \item[(1)] Artin Conjecture is true at $s_0$: $\mathbb F[H(s_0)]=\mathbb F[x_1,\ldots,x_r]$.
 \item[(2)] $I_{H(s_0)}=(0)$ and for any $1\leq t\leq r-1$ we have that $L_t\geq N_t$.
 \item[(3)] $I_{H(s_0)}=(0)$ and there exists $1\leq t\leq r-1$ such that $L_t\geq N_t$.
\end{enumerate}
\end{teor}

\begin{proof}
$(1)\Rightarrow (2)$ Since $\mathbb F[H(s_0)]=\mathbb F[x_1,\ldots,x_r]$, it follows that $L_t=\binom{r}{t}\geq N_t$
                     for any $1\leq t\leq r-1$. $(2)\Rightarrow (3)$ It is obvious.

$(3)\Rightarrow (1)$ We assume that $\mathbb F[H(s_0)]\neq \mathbb F[x_1,\ldots,x_r]$. According to Theorem $1.4$, 
                     there exists $2\leq q\leq r$ and $m_2,\ldots,m_q>0$ such that
										$$\mathbb F[H(s_0)]=\mathbb F[x_1,x_1^{m_2}x_2,\ldots,x_1^{m_q}x_q,x_{q+1},\ldots,x_r].$$
										Let $1\leq t\leq r-1$ such that $L_t\geq N_t$. If $v\in \mathbb F[H(s_0)]$ is a monomial with 
										$|\supp(v)|=t$, then either $x_1\in\supp(v)$, either $x_1\notin\supp(v)$ and $\supp(v)\subset\{x_{q+1},\ldots,x_r\}$.
										It follows that 
									  $$L_t = \binom{r-1}{t-1} + \binom{r-q}{t} \leq \binom{r-1}{t-1} + \binom{r-2}{t} < N_t,$$
										a contradiction.
\end{proof}

The following Corollary is a restatement of \cite[Theorem 2]{numb} and \cite[Theorem 1]{numb2}:

\begin{cor}
 The following assertions are equivalent:
\begin{enumerate}
 \item[(1)] Artin Conjecture is true at $s_0$: $\mathbb F[H(s_0)]=\mathbb F[x_1,\ldots,x_r]$.
 \item[(2)] $I_{H(s_0)}=(0)$ and $\prod_{j\neq i} x_j \in \mathbb F[H(s_0)]$, $(\forall)1\leq i\leq r$.
 \item[(3)] $I_{H(s_0)}=(0)$ and for every $j,k \in \{1,\ldots,r\}$, $j \neq k$ there exists a monomial $v\in \mathbb F[H(s_0)]$ such that $x_j|v$ and $x_{k}\nmid v$.
 \item[(4)] $I_{H(s_0)}=(0)$ and there exists $1\leq t \leq r-1$ such that for every set $M\subset\{1,\ldots,r\}$ with $t$ elements and for every $j\in M$ there exists $k_j>0$ such that $\prod_{j\in M} x_j^{k_j} \in \mathbb F[H(s_0)]$.
\end{enumerate}
\end{cor}

\begin{proof}
$(1)\Rightarrow (2),(3),(4)$ are obvious. 

$(2)\Rightarrow (1)$ The hypothesis $(2)$ implies $L_{r-1}=\binom{r}{r-1}=r=N_{r-1}$, hence the conclusion follows from Theorem $1.6$.

$(4)\Rightarrow (1)$ The hypothesis $(4)$ implies that there exists $1\leq t\leq r-1$ such that $L_t=\binom{r}{t}\geq N_t$, hence the conclusion follows from Theorem $1.6$.

$(3)\Rightarrow (1)$ Assume $\mathbb F[H(s_0)]\neq \mathbb F[x_1,\ldots,x_r]$ and let $v\in \mathbb F[H(s_0)]$ be a monomial. 
From Theorem $1.4$, it follows that $x_2|u$ implies $x_1|u$, hence we contradict $(4)$.
\end{proof}

\begin{cor}
 If $f_1,\ldots,f_r$ have simple zeros distinct, then, for any $s_0\in \mathbb C\setminus\{1\}$, the following are equivalent.
\begin{enumerate}
 \item[(1)] Artin's Conjecture is true at $s_0$: $\mathbb F[H(s_0)]=\mathbb F[x_1,\ldots,x_r]$.
 \item[(2)] There exists $1\leq t\leq r-1$ such that $L_t\geq N_t$.
 \item[(3)] $\prod_{j\neq i} x_j \in \mathbb F[H(s_0)]$, $(\forall)1\leq i\leq r$.
 \item[(4)] For every $j,k \in \{1,\ldots,r\}$, $j \neq k$ there exists a monomial $u\in \mathbb F[H(s_0)]$ such that $x_j|u$ and $x_{k}\nmid u$.
 \item[(5)] There exists $1\leq m < r$ such that for every set $M\subset\{1,\ldots,r\}$ with $m$ elements and for every $j\in M$ there exists $k_j>0$ such that $\prod_{j\in M} x_j^{k_j} \in \mathbb F[H(s_0)]$.
\end{enumerate}
\end{cor}

\begin{proof}
It follows from Theorem $1.4$, Theorem $1.6$ and Corollary $1.7$.
\end{proof}

A finite group $G$ is called \emph{almost monomial} if for every distinct irreducible characters $\chi$ and $\psi$
of $G$ there exist a subgroup $H$ of $G$ and a linear character $\varphi$ of $H$ such
that the induced character $\varphi^G$ contains $\chi$ and does not contain $\psi$. Every
monomial group and every quasi monomial group in the sense of \cite{numb0} are almost monomial.
We restate \cite[Theorem 1]{monat}:

\begin{teor}
If the group $G$ is almost monomial and $I_{H(s_0)}=(0)$ then the Artin conjecture holds at $s_0$.
\end{teor}

\begin{cor}
If the group $G$ is almost monomial and $f_1,\ldots,f_r$ have simple zeros distinct, then the Artin conjecture holds
for any $s_0\in \mathbb C\setminus\{1\}$.
\end{cor}

\begin{proof}
It follows from Theorem $1.4$ and Theorem $1.9$.
\end{proof}

In the following, we discuss other cases, in which $I_{H(s_0)}\neq 0$.

\begin{prop}
Assume that $f_{p+1},\ldots,f_q$ have simple poles at $s_0$. Then:
$$\mathbb F[H(s_0)] = \mathbb F[x_{q+1},\ldots,x_r, x_j v\;: 1\leq j\leq p,\;\supp(v)\subset \{x_{p+1},\ldots,x_q\}, \deg(v)\leq \ell_j],$$
where $\supp(v):=\{x_j\;:\;x_j|v\}$ is the \emph{support} of the monomial $v$.
\end{prop}

\begin{proof}
Let $R:=\mathbb F[x_{q+1},\ldots,x_r, x_j v\;: 1\leq j\leq p,\;\supp(v)\subset \{x_{p+1},\ldots,x_q\}, \deg(v)\leq \ell_j]$.
Since $\ell_{p+1}=\cdots=\ell_q=-1$, it follows that $\ord(x_j v) = \ell_j - \deg(v)$ for any $1\leq j\leq p$. 
Hence $R\subseteq \mathbb F[H(s_0)]$. In order to prove the converse inclusion, 
let $u=x_1^{a_1}x_2^{a_2}\cdots x_r^{a_r}\in \mathbb F[H(s_0)]$. 
It follows that
$$\ord(u) = a_1 \ell_1 + \cdots a_p \ell_p  -  a_{p+1} - \cdots - a_q \geq 0.$$
We prove that $u\in R$, using induction on $\deg(u)\geq 0$.
If $\deg(u)=0$, i.e. $u=1$, then there is nothing to prove. Assume $\deg(u)>0$. If $a_{p+1}=\cdots=a_q=0$ then $u\in R$.
If this is not the case, then there exists $1\leq j\leq p$ such that $a_j>0$. We choose a monomial $v$ such that
$$\deg(v) = \min\{ a_{p+1}+\cdots+a_q, \ell_j\},\;\supp(v)\subset \{x_{p+1},\ldots,x_q\},\;\deg_{x_j}(v)\leq a_j,\;p+1\leq j\leq q,$$
where $\deg_{x_j}(v)=\max\{k:\;x_j^k|v\}$. We write $u=u'v$. Since $\ord(u')\geq 0$, by induction hypothesis, it follows that $u'\in R$,
hence $u=u'(x_jv) \in R$.
\end{proof}

\begin{prop}
Assume that $f_{1},\ldots,f_p$ have simple zeros at $s_0$. Then: \small{
$$\mathbb F[H(s_0)] = \mathbb F[x_1,\ldots,x_p, x_{q+1},\ldots,x_r, v x_j\;: p+1\leq j\leq q,\;\supp(v)\subset \{x_{1},\ldots,x_p\}, \deg(v) = - \ell_j].$$}
\end{prop}

\begin{proof}
We denote $$R:=\mathbb F[x_1,\ldots,x_p, x_{q+1},\ldots,x_r, v x_j\;: p+1\leq j\leq q,\;\supp(v)\subset \{x_{1},\ldots,x_p\}, 
\deg(v) = - \ell_j].$$ 
Since $\ell_{1}=\cdots=\ell_p=1$, it follows that $\ord(x_j v) = \ell_j + \deg(v) = 0$ for any $p+1\leq j\leq q$
and any monomial $v\in \mathbb F[x_1,\ldots,x_p]$ with $\deg(v)=-\ell_j$. Hence $R\subseteq \mathbb F[H(s_0)]$. 
In order to prove $\supseteq$, let $u=x_1^{a_1}x_2^{a_2}\cdots x_r^{a_r}\in \mathbb F[H(s_0)]$. It follows that
$$ \ord(u) = a_1+a_2+\cdots+a_p+\ell_{p+1}a_{p+1}+\cdots \ell_q a_q \geq 0. $$
Let $s:=- \ell_{p+1}a_{p+1} - \cdots - \ell_q a_q$. We prove that $u\in R$, using induction on $s\geq 0$.
If $s=0$ then we choose a monomial $v\in \mathbb F[x_1,\ldots,x_p]$ such that $\deg(v) = - \ell_j$ and
$\deg_{x_j}(v)\leq a_j,\;1\leq j\leq p$. We write $u=u'\cdot(x_jv)$. Applying the induction hypothesis on $u'$,
it follows that $u'\in R$ and thus $u\in R$, as required.
\end{proof}

\begin{teor}
Assume that $f_1,\ldots,f_r$ have at most simple zeros and simple poles at $s_0$. If Artin conjecture does not hold at $s_0$, then:
\begin{enumerate}
\item[(1)] $\mathbb F[H(s_0)] = \mathbb F[x_1,\ldots,x_{p},x_{q+1},\ldots,x_r, x_{j}x_k \;:\; 1\leq j\leq p,\;p+1 \leq k\leq q]$.
\item[(2)] Letting $\Phi:\mathbb F[t_1,\ldots,t_p,t_{q+1},\ldots,t_r, t_{jk}\;:\; 1\leq j\leq p,\;p+1 \leq k\leq q] \rightarrow \mathbb F[H(s_0)]$,
           $\Phi(t_j)=x_j$, $\Phi(t_{jk})=x_jx_k$, we have:
$$I_{H(s_0)}=\Ker(\Phi)=(t_jt_{ik}-t_it_{jk},\; t_{jk}t_{im} - t_{jm}t_{ik} \;:\; 1\leq j,i \leq p,\;p+1\leq k,m\leq q).$$
\end{enumerate}
\end{teor}

\begin{proof}
(1) It follows from Proposition $1.11$ and Proposition $1.12$.

(2) Let $I=(t_jt_{ik}-t_it_{jk},\; t_{jk}t_{im} - t_{jm}t_{ik} \;:\; 1\leq j,i \leq p,\;p+1\leq k,m\leq q)$.
    The inclusion $I\subseteq I_{H(s_0)}$ is obvious. Conversely, let $v\in I_{H(s_0)}$, $v=v^+ - v^-$ be a binomial.
		It follows that 
		\begin{equation}\label{pool}
		\Phi(v^+)=\Phi(v^-)=x_1^{a_1}\cdots x_r^{a_r}, \text{ where } a_j\geq 0,\;1\leq j\leq r. 
		\end{equation}
		Let \begin{equation}\label{poool}
		v^+ = t_1^{b_1}\cdots t_p^{b_p}t_{q+1}^{b_{q+1}}\cdots t_r^{b_r} \prod_{\substack{1\leq j \leq p \\ p+1 \leq k\leq q}}t_{jk}^{b_{jk}},
		\; v^- = t_1^{b'_1}\cdots t_p^{b'_p}t_{q+1}^{b'_{q+1}}\cdots t_r^{b_r} 
		\prod_{\substack{1\leq j \leq p \\ p+1 \leq k\leq q}}t_{jk}^{b'_{jk}}
		\end{equation}                
		From (\ref{pool}) and (\ref{poool}) it follows that \small{ 
                \begin{equation}\label{poola}
                b_j + \sum_{k=p+1}^q b_{jk} = b'_j + \sum_{k=p+1}^q b'_{jk},\; 1\leq j\leq p, \; 
		   \sum_{j=1}^p b_{jk} = \sum_{j=1}^p b'_{jk},\;p+1\leq k\leq q,\; b_j=b'_j,\; q+1\leq j\leq r. 
                \end{equation}}		
                Consequentely, we can assume that $b_j=b'_j=0$, $(\forall)q+1\leq j\leq r$.

                Let $s(v):=b_1+\cdots+b_p-b'_1-\cdots-b'_p$. Without any loss of generality, we assume $s(v)\geq 0$.
                We claim that we can assume $b_j\leq b'_j$, $(\forall)1\leq j\leq p$.
                Indeed, if let's say $b_1<b'_1$ then there exists $p+1\leq k\leq q$ such that $b_{1k} > b'_{1k}$. 
                Also, as $s(v)\geq 0$, then there exists $2\leq j\leq p$ such that $b_j>b'_j$.
                Since $t_1t_{jk}-t_jt_{1k}\in I_{H(s_0)}$, it follows that $v\in I_{H(s_0)}$ if and only if 
                $v' = \frac{t_1t_{jk}}{t_jt_{1k}}v^+ - v^{-} \in I_{H(s_0)}$. As $s(v)=s(v')$, using repeatedly the previous argument,
                we prove the claim. 

                Consequently, without any loss of generality, we can assume $b'_j=0$, $(\forall)1\leq j\leq p$.
                From (\ref{poola}) it follows that $$ b_j + \sum_{k=p+1}^q b_{jk} = \sum_{k=p+1}^q b'_{jk},\; 1\leq j\leq p, \; 
		   \sum_{j=1}^p b_{jk} = \sum_{j=1}^p b'_{jk},\;p+1\leq k\leq q, $$
                which implies $b_j=0$, $(\forall)1\leq j\leq p$. Therefore, we have 
                \begin{equation}\label{piulita}
                 \sum_{k=p+1}^q b_{jk} = \sum_{k=p+1}^q b'_{jk},\; 1\leq j\leq p, \; 
		  \sum_{j=1}^p b_{jk} = \sum_{j=1}^p b'_{jk},\;p+1\leq k\leq q,
                \end{equation}
                $$v=\prod_{j=1}^p\prod_{k=p+1}^q t_{jk}^{b_{jk}}=\prod_{j=1}^p\prod_{k=p+1}^q t_{jk}^{b'_{jk}}.$$
                We claim that we can assume $b_{jk}\geq b'_{jk}$, $(\forall)1\leq j\leq p,\;p+1\leq k\leq q$, hence, by (\ref{piulita}), $b_{jk}=b'_{jk}$, $(\forall) 1\leq j\leq p,\;p+1\leq k\leq q$ 
                which implies $v=0\in I$, as required. Indeed, if this is not the case, we have $b_{jk}<b'_{jk}$ for some $1\leq j\leq p$ and $p+1\leq k\leq q$. From (\ref{piulita}) it follows
                that there exists some indices $1\leq m\leq p$, $p+1\leq n\leq q$ such that $j\neq m$, $k\neq n$ and $b_{mn}<b'_{mn}$. Since $t_{jk}t_{mn} - t_{jn}t_{km} \in I_{H(s_0)}$, it follows that
                $v\in I_{H(s_0)}$ if and only if $v'=\frac{t_{jk}t_{mn}}{t_{jn}t_{km}}v^+ - v^- \in I_{H(s_0)}$. Using repeatedly the same argument, we prove the claim.
\end{proof}

\section{Virtual characters}
Let $\chi = a_1\chi_1 + a_2\chi_2 + \cdots + a_r\chi_r,\;a_j\in\mathbb Z,\;1\leq j\leq r$, be a virtual character of the group $G=\Gal(K/\mathbb Q)$,
and let $L(s,\chi):=f_1^{a_1}f_2^{a_2}\cdots f_r^{a_r}$. The map
\begin{equation}\label{cuuur}
\overline{Ar} \rightarrow \mathbb Z^r, \; L(s,\chi)\mapsto (a_1,\ldots,a_r),
\end{equation}
is an isomophism of groups, where $\overline{Ar}$ was defined in Introduction. Since $L(s,\chi_1),\ldots,L(s,\chi_r)$ are algebraically independent over $\mathbb F$,
it follows that 
$$ \mathbb F[\overline{Ar}] \cong \mathbb F[x_1^{\pm 1},\ldots,x_r^{\pm 1}] = \text{ the ring of Laurent polynomial in } r \text{ indeterminates }.$$
We consider the set 
$$\overline{Hol(s_0)} = \{f=L(s,\chi)\;:\;f \text{ holomorphic at }s_0,\;\chi \text{ a virtual character of }G\}.$$
One can easily check that $\overline{Hol(s_0)}$ is a subsemigroup in $\overline{Ar}$ and (\ref{cuuur}) gives an isomorphism
$$ \overline{Hol(s_0)} \cong \overline{H(s_0)}:=\{(a_1,\ldots,a_r)\in\mathbb Z^r\;:\; a_1\ell_1+\cdots+a_r\ell_r\geq 0\}. $$
By Gordon's Lemma, the semigroup $\overline{H(s_0)}$ is affine. It follows that
$$\mathbb F[\overline{Hol(s_0)}] \cong \mathbb F[\overline{H(s_0)}] \subseteq \mathbb F[x_1^{\pm 1},\ldots,x_r^{\pm 1}],$$
and $\mathbb F[\overline{H(s_0)}]$ is a toric ring.
Since $\overline{H(s_0)}\cap \mathbb N^r = H(s_0)$, it follows that
\begin{equation}\label{curva}
  \mathbb F[\overline{H(s_0)}]\cap \mathbb F[x_1,\ldots,x_r] = \mathbb F[H(s_0)].
\end{equation}
Let $u=x_1^{a_1}\cdots x_r^{a_r}\in \mathbb F[x_1^{\pm 1},\ldots,x_r^{\pm r}]$, where $a_j\in\mathbb Z$, $1\leq j\leq r$. We denote
$$\ord(u):=a_1\ell_1+\cdots+a_r\ell_r.$$
It follows that $u\in \mathbb F[\overline{H(s_0)}]$ if and only if $\ord(u)\geq 0$. 

\begin{obs}\emph{
In \cite[p. 871]{heil} Heilbronn related Artin's Conjecture with the language of character theory. He introduced the virtual character
$\Theta:= \ell_1\chi_1 + \cdots + \ell_r \chi_r$,
with the corresponding function $L(\Theta,s)=f_1^{\ell_1}\cdots f_r^{\ell_r}$, which is holomorphic at $s_0$, i.e. $L(\Theta,s)\in \overline{Hol(s_0)}$. 
Note that $\ord_{s=s_0}L(\Theta,s) = \sum_{j=1}^r|\ell_j|$. Let
 $m_{\Theta}=x_1^{\ell_1}\cdots x_r^{\ell_r}\in \mathbb F[x_1^{\pm 1},\ldots,x_r^{\pm 1}]$.
If is easy to note that Artin conjecture is true at $s_0$ if and only if $m_{\Theta} \in \mathbb F[x_1,\ldots,x_r]$.
A recent generalization of Heilbronn character can be found in \cite{wong}.}
\end{obs}

Our aim is to describe the generators of the $\mathbb F$-algebra $\mathbb F[\overline{H(s_0)}]$.
Using the the assumption (\ref{lab5}) and the notations from Theorem $1.4$, we prove the following result:

\begin{prop}
Assume that Artin conjecture fails at $s_0$. The following are equivalent:
\small{
\begin{enumerate}
 \item[(1)] $I_{H(s_0)}=(0)$.
 \item[(2)] $\mathbb F[\overline{H(s_0)}]= \mathbb F[x_1,(x_1^{m_2}x_2)^{\pm 1},\ldots,(x_1^{m_q}x_q)^{\pm 1}, x_{q+1}^{\pm 1},\ldots, x_r^{\pm 1}]$,
where $m_j=-\frac{\ell_j}{\ell_1}$, $2\leq j\leq q$.
\end{enumerate}}
\end{prop}

\begin{proof}
$(1)\Rightarrow (2)$ Let $R:=\mathbb F[x_1, (x_1^{m_2}x_2)^{\pm 1},\ldots,(x_1^{m_q}x_q)^{\pm 1}, x_{q+1}^{\pm 1},\ldots, x_r^{\pm 1}]$.
    Since $$\ell_1>0,\; \ell_2,\ldots,\ell_q<0,\; \ell_{q+1}=\cdots=\ell_r=0, \ord(x_1^{m_j}x_j)=\ell_1m_j - \ell_j = 0, (\forall)2\leq j\leq q,$$
     it follows that
    $R\subseteq \mathbb F[\overline{H(s_0)}]$. In order to prove $\supseteq$, let  $u=x_1^{a_1}\cdots x_r^{a_r}\in \mathbb F[\overline{H(s_0)}]$,  where $a_j\in\mathbb Z$, $1\leq j\leq r$.
    Since $\ell_{q+1}=\ldots=\ell_r=0$ and $x_j^{\pm 1}\in R$, $(\forall)q+1\leq j\leq r$, it is enough to prove that
    $$u':=x_1^{a_1}x_2^{a_2}\cdots x_q^{a_q}\in R':=\mathbb F[x_1,(x_1^{m_2}x_2)^{\pm 1},\ldots,(x_1^{m_q}x_q)^{\pm 1}].$$
    Note that $\ord(u)=\ord(u')=\ell_1(a_1-m_2a_2-\cdots-m_qa_q)\geq 0$, therefore 
\begin{equation}\label{cucuc}
a_1\geq m_2a_2-\cdots-m_qa_q. 
\end{equation}
    Without any loss of generality, we can assume that there exists $1\leq s\leq q$ such that $a_2,\ldots,a_s<0$ and $a_{s+1},\ldots,a_q\geq 0$.
    One can easily see that $u'\in R'$ if and only if 
    $$u'':= (x_1^{m_2}x_2)^{-a_2}\cdots (x_1^{m_s}x_s)^{-a_s}u' = x_1^{a_1-m_2a_2-\cdots-m_sa_s}x_{s+1}^{a_{q+1}}\cdots x_q^{a_q} \in R'$$
    On the other hand, from (\ref{curva}) and (\ref{cucuc}) it follows that 
		$u''\in \mathbb F[H(s_0)]\cap \mathbb F[x_1,\ldots,x_q] \subset R'$, as required.

$(2)\Rightarrow (1)$. From (\ref{curva}) it follows that 
$\mathbb F[H(s_0)] = \mathbb F[x_1,x_1^{m_2}x_2,\ldots,x_1^{m_q}x_q,x_{q+1},\ldots,x_r]$, hence, as in the proof of Theorem $1.4$, $I_{H(s_0)}=(0)$.
\end{proof}

Using the assumption $(\ref{lab5})$, we prove the following:

\begin{prop}
 Assume that $f_1,\ldots,f_r$ have only simple zeros or poles at $s_0$. If Artin conjecture does not hold at $s_0$, then
 $$\mathbb F[\overline{H(s_0)}] = \mathbb F[x_1,\ldots,x_{p},x_{q+1}^{\pm 1},\ldots,x_r^{\pm 1}, (x_{j}x_k)^{\pm 1} \;:\; 1\leq j\leq p,\;p+1 \leq k\leq q].$$
\end{prop}

\begin{proof}
Let $R:=\mathbb F[x_1,\ldots,x_{p},x_{q+1}^{\pm 1},\ldots,x_r^{\pm 1}, (x_{j}x_k)^{\pm 1} \;:\; 1\leq j\leq p,\;p+1 \leq k\leq q]$. It is easy to check
that $R\subset \mathbb F[\overline{H(s_0)}]$. Let $u=x_1^{a_1}\cdots x_r^{a_r}\in \mathbb F[\overline{H(s_0)}]$,  where $a_j\in\mathbb Z$, $1\leq j\leq r$.
We have that $\ord(u)=a_1+\cdots+a_p - a_{p+1} -\cdots - a_q \geq 0$. Let $u':=x_1^{a_1}\cdots x_q^{a_q}$. Since $x_j^{\pm 1}\in R$, $(\forall)q+1\leq j\leq r$,
then $u\in R$ if and only if $$u'\in R':=\mathbb F[x_1,\ldots,x_{p},(x_{j}x_k)^{\pm 1} \;:\; 1\leq j\leq p,\;p+1 \leq k\leq q].$$
As in the proof of Proposition $2.2$, we can assume that $a_1,\ldots,a_q\geq 0$. From Theorem $1.13$, it follows that $u'\in \mathbb F[H(s_0)]\cap \mathbb F[x_1,\ldots,x_q] \subset R'$
as required.
\end{proof}

{}

\vspace{2mm} \noindent {\footnotesize
\begin{minipage}[b]{15cm}
Mircea Cimpoea\c s, Simion Stoilow Institute of Mathematics, Research unit 5, P.O.Box 1-764,\\
Bucharest 014700, Romania, E-mail: mircea.cimpoeas@imar.ro
\end{minipage}}

\end{document}